\documentclass{amsart}

\usepackage{amssymb}
\usepackage{graphicx, epsfig}
\usepackage{latexsym, amsfonts, amscd, amsmath}

\theoremstyle{plain}
\newtheorem{thm}{Theorem}[section]
\newtheorem{prop}[thm]{Proposition}

\newtheorem{cor}[thm]{Corollary}

\theoremstyle{definition}

\theoremstyle{remark}

\newtheorem{rmk}[thm]{Remark}

\catcode`\Ž=\active\def Ž{\'e}
\catcode`\=\active\def {\`e}
\catcode`\ˆ=\active\def ˆ{\`a}
\catcode`\‰=\active\def ‰{\^a}
\catcode`\Ë=\active\def Ë{\`A}
\catcode`\Š=\active\def Š{\"a}
\catcode`\=\active\def {\`u}
\catcode`\=\active\def {\^{e}}
\catcode`\"=\active\def "{\^{i}}
\catcode`\•=\active\def •{\"i}
\catcode`\™=\active\def ™{\^{o}}
\catcode`\š=\active\def š{\"{o}}
\catcode`\ž=\active\def ž{\^{u}}
\catcode`\Ÿ=\active\def Ÿ{\"{u}}
\catcode`\†=\active\def †{\"{U}}
\catcode`\=\active\def {\c{c}}
\catcode`\'=\active\def '{\c{C}}
\catcode`\é=\active\def é{\c{C}}

\newcommand{\End}{{\operatorname{End}}}

\newcommand{\Hom}{{\operatorname{Hom}}}

\newcommand{\SL}{{\operatorname{SL }}}

\newcommand{\Gal}{{\operatorname{Gal}}}

\newcommand{\Lift}{{\operatorname{Lift}}}

\DeclareMathOperator{\SK}{SK}
\newcommand{\tr}{{\mathrm{tr}}}
\newcommand{\smallmat}[4]{\bigl(\begin{smallmatrix}#1&#2\\#3&#4\end{smallmatrix}\bigr)}

\newcommand{\gera}{{\mathfrak{a}}}
\newcommand{\gerb}{{\mathfrak{b}}}

\newcommand{\gerd}{{\mathfrak{d}}}

\newcommand{\gerf}{{\mathfrak{f}}}

\def\C{\mathbb{C}}

\def\Q{\mathbb{Q}}

\def\Z{\mathbb{Z}}

\def\mat#1#2#3#4{\left( \begin{array}{cc} #1 & #2 \\ #3 & #4 \end{array} \right) }

\def\2vector#1#2{\left( \begin{smallmatrix} #1 \\ #2 \end{smallmatrix}
\right)}

\def\deb{ \begin{equation} }
\def\fin{ \end{equation} }

\usepackage[usenames]{color}
\definecolor{Indigo}{rgb}{0.2,0.1,0.7}
\definecolor{Violet}{rgb}{0.5,0.1,0.7}
\definecolor{White}{rgb}{1,1,1}
\definecolor{Green}{rgb}{0.1,0.9,0.2}

\begin{document}

\title[Saito-Kurokawa and Darmon points]{The Saito-Kurokawa lifting and Darmon points}
\date{}
\author{Matteo Longo, Marc-Hubert Nicole}

\begin{abstract} Let $E_{/_\Q}$ be an elliptic curve of conductor $Np$ with $p\nmid N$ 
and let $f$ be its associated newform of weight $2$. Denote by $f_\infty$ the 
$p$-adic Hida family passing though $f$, and by $F_\infty$ its 
$\Lambda$-adic Saito-Kurokawa lift. The $p$-adic family $F_\infty$ of Siegel modular forms 
admits a formal Fourier expansion, from which we can define
a family of normalized Fourier coefficients $\{\widetilde A_T(k)\}_T$ indexed by 
positive definite symmetric half-integral matrices $T$ of size $2\times 2$. 
We relate explicitly certain global points on $E$ (coming from the theory of Darmon points) with the values  
of these Fourier coefficients and of their $p$-adic derivatives, evaluated at weight $k=2$. 
\end{abstract}

\address{M.L. Dipartimento di Matematica Pura e Applicata, Universit\`a di Padova, Via Trieste 63, 35121 Padova, Italy}
\email{mlongo@math.unipd.it}

\address{M.H.N. Institut de mathŽmatiques de Luminy, UniversitŽ d'Aix-Marseille, campus de Luminy, case 907, 13288 Marseille cedex 9, France}
\email{nicole@iml.univ-mrs.fr}

\maketitle

\section{Introduction}\label{intro}
Let $f$ be an elliptic newform of weight $2$ and level $\Gamma_0(M)$. 
Eichler and Shimura showed how to associate to $f$ an abelian variety $A_f$ of arithmetic conductor $M$ such that the complex $L$-functions 
attached to $f$ and $A_f$ agree, cf. \cite[Theorem 7.14]{Sh1}.
Although the theory of Siegel modular forms provides a satisfactory generalization of the notion of classical elliptic modular forms in higher dimension, 
no such construction is known for Siegel modular forms of genus $>1$. Generalizing the Shimura-Taniyama conjecture, Yoshida \cite{Yo} conjectured the existence of a genus two holomorphic Siegel modular cusp eigenform of parallel weight $2$ associated to any irreducible abelian surface $A$ defined over $\Q$. Note that this Siegel modular form, in contrast with the elliptic case, cannot be obtained in general by cutting out a piece of the Žtale cohomology of the Siegel modular variety. This makes the connection to geometric constructions much less immediate. 

In this paper, we study instead the reducible case by considering Siegel cusp forms in the 
image of the Saito-Kurokawa lifting.
We use this lift to convert known 
connections between elliptic modular forms 
and certain global points on rational elliptic curves to Siegel modular forms
in the image of the Saito-Kurokawa lifting. 
Our main result (partially summarized in Theorem \ref{main-intro} below) 
provides a relation 
between Fourier coefficients of the Saito-Kurokawa lift of $f$ and global
points on the elliptic curve associated with $f$ as above. This result offers
a geometric interpretation of Fourier coefficients of Saito-Kurokawa lifts.
It is obtained by combining $p$-adic techniques 
developed by Darmon-Tornar\'\i a in \cite{DT} for the $\Lambda$-adic Shintani lifting 
with various results on the Saito-Kurokawa lifting recently obtained  
by Ibukiyama in \cite{Ibu} and an explicit description of Fourier coefficients of modular forms of half-integral 
weight by Kohnen \cite{Ko2}. The global points providing such a description come from the theory of Stark-Heegner points 
introduced by Darmon in \cite{Dar}. Recently, after \cite{LRV}, there has been a move in the literature to relabel Stark-Heegner points 
as Darmon points, explaining the title of the paper. 

To describe our work and main results in a more precise form, 
we fix an elliptic newform  $f$ of weight 2, level $\Gamma_0(Np)$ and rational Fourier coefficients, where $N>1$ is an integer and $p\nmid N$ is a fixed prime number. 

Denote by 
\begin{equation}\label{Hida-family}
f_\infty(k)=\sum_{n\geq 1}a_n(k)q^n\end{equation} the Hida family passing through $f$, 
where $k$ is a $p$-adic variable in a neighborhood $U$ of $2$ in 
\[
\mathcal X:=\Hom_{\rm cont}(\Z_p^\times,\Z_p^\times)\]
and $a_1(k)=1$ for all $k\in U$.
Here $\Z\hookrightarrow \mathcal X$ via $k\mapsto(x\mapsto x^k)$. For simplicity, we will also 
assume that $U$ is contained in the residue class of 2 modulo $p-1$. Thus $f_\infty(2)$ coincides 
with the Fourier expansion of $f$ and, more generally,  $f_\infty(k)$ 
is a normalized ordinary eigenform of level $\Gamma_0(Np)$ for all positive integers $k\in U$. 
Using the explicit $\Lambda$-adic Saito-Kurokawa lifting following Li \cite{Li} and Kawamura \cite{Ka}
and based on work of Stevens \cite{St}, 
in Section \ref{sectionone} we recall the construction of a $p$-adic family of Siegel modular forms 
\[F_\infty(k):=\sum_{T>0}\gera_T(k)q^T\] (where $k\in U$ and 
the sum runs over all positive definite, half-integral symmetric 
matrices $T$ of size $2 \times 2$) 
interpolating the $p$-stabilization of the Saito-Kurokawa lifting of the classical forms appearing 
in the Hida family $f_\infty(k)$. In particular,  
$F_\infty(2)$ is just the Fourier expansion of the Saito-Kurokawa lift of $f$ (well-defined up to 
non-zero complex factors).  
In Section \ref{sectiontwo},
we introduce a normalization of the coefficients $\gera_T(k)$, which we denote $\widetilde A_T(k)$;
these are $p$-adic analytic functions on a neighborhood $U$ of $2$ in $\mathcal X$.
The analysis of the modified Fourier coefficients $\widetilde A_T(k)$
is carried out in Section \ref{sectiontwo}, which contains the technical heart of the proof. We combine various results for the explicit 
Saito-Kurokawa lifting due to Ibukiyama \cite{Ibu} and explicit relations for Fourier coefficients 
on modular forms of half-integral weight due to Kohnen \cite{Ko2}.  In particular, 
the paper \cite{Ibu} generalizes to arbitrary level 
various results on the Saito-Kurokawa lifting due to, among others, Eichler-Zagier \cite{EZ}, Kohnen \cite{Ko}, 
Manickam-Ramakrishnan-Vasudevan \cite{MRV} and  Manickam-Ramakrishnan \cite{MR}.  
In Section \ref{sectionthree}, we combine our explicit analysis of the modified Fourier 
coefficients $\widetilde A_T(k)$ with 
the work of Darmon-Tornar\'\i a \cite{DT} on half-integral weight modular forms 
to relate the values $\widetilde A_T(2)$ 
and $\widetilde A'_T(2)$ to global points 
on the elliptic curve $E$ associated with $f$
via the Eichler-Shimura construction. To be more precise, the family 
of $p$-adic Fourier coefficients $\{\widetilde A_T(k)\}$ can be divided into two subfamilies called of type I and II, respectively, 
corresponding to those $T$ for which $\widetilde A_T(2)$ do not need to vanish
and respectively to those $T$ such that $A_T(2)$ is forced to vanish, see
\S \ref{sec4.1} and eq. \eqref{D_T} for precise definitions. 
Thus, for $T$ of type II we have $A_T(2)=0$ 
and it natural to look at the value of the first derivative 
$A'_T(2)$ (with respect to $k$) at $k=2$. Our main result, corresponding to 
Theorem \ref{thm1} below, relates this derivative to a global point on the elliptic curve $E$, 
defined over a quadratic imaginary field depending on $T$. To state the result more precisely, we need 
some further notations. Let 
\begin{equation}\label{Phi-Tate}
\Phi_{\rm Tate}: \C_p^\times/q^\Z\longrightarrow E(\C_p)\end{equation} be Tate's $p$-adic uniformization, 
and let 
\[\log_E:E(\C_p)\longrightarrow\C_p\] 
be the $p$-adic formal group logarithm, defined by 
\begin{equation}\label{log}
\log_E(P):=\log_q(\Phi_{\rm Tate}^{-1}(P)),\end{equation}
where
$\log_q$ is the branch of the $p$-adic logarithm satisfying $\log_q(q)=0$.
Extend $\log_E$ by $\Q$-linearity to $E(\C_p)\otimes_\Z\Q$. 
Finally, say that $T=\smallmat u{v/2}{v/2}w$ is \emph{primitive} 
if 
$\gcd(u,v,w)=1$. 
Our main result can now be stated as follows. 

\begin{thm} \label{main-intro}
There exists a point $Q_T\in E(K_T)\otimes_\Z\Q$, 
where $K_T$ is an imaginary quadratic field depending on $T$, 
such that 
\begin{equation}\label{thm1-a}
\frac{\partial}{\partial k}\widetilde A_T(k)_{|k=2} = 
\log_E (Q_T).\end{equation}
Further, if $T$ is primitive, 
\begin{equation}\label{thm1-b}Q_T\neq 0\Longleftrightarrow L'(\SK(f),\chi_{K_T},1)\neq 0,
\end{equation} where $L(\SK(f),\chi_{K_T},s)$ is the Adrianov $L$-function attached to the Saito-Kurokawa 
lift $\SK(f)$ of $f$, twisted by the quadratic character $\chi_{K_T}$ of $K_T$. \end{thm}

The above equations \eqref{thm1-a} and \eqref{thm1-b} in Theorem \ref{main-intro}
may be viewed as an analogue of \cite[Theorems 1.5, 5.1]{DT} 
where similar results are established for the Shintani lifting in lieu of the Saito-Kurokawa lifting.

We also remark that a more general version 
\eqref{thm1-b} above is proven in the text, relaxing the primitivity assumption on $T$. 
For more general $T$ of type II, the relation 
\eqref{thm1-b} 
between the global point 
$Q_T$ and the $L$-function $L(\SK(f),\chi_{K_T},s)$ depends on 
the value of a certain explicit integer attached to $T$,  
denoted $n_T$ in the text. 
If $n_T=0$, the condition $L'(\SK(f),\chi_{K_T},1)\neq0$ does not 
imply the non-vanishing of the point $Q_T$. See Remark \ref{remark-n_T} for an explicit 
description of $n_T$ and a discussion on its possible vanishing. 

Our second main result, Theorem \ref{thm2}, deals instead with $T'$ of type I. 
Having fixed a $T$ of type II, we show that, at least 
when $T'$ is primitive of type I, the product 
$\widetilde A_{T'}(2)\log_E(Q_T)$ is equal to a quantity  
(denoted by $J(f,\gerd_T,\gerd_{T'})$ in the text) obtained as a
$p$-adic variation of classical Shintani geodesic integrals attached to $f$ 
and certain quadratic forms depending on $T$ and $T'$,
see \eqref{shintani} for details and definitions. This result also holds without the assumption, made 
in the introduction, that $N>1$. 

We mention that Kawamura    
investigated in \cite{Ka} applications of $p$-adic methods ˆ la Hida to the Ikeda lifting, introducing more general $p$-stabilized families of Siegel modular forms.
Also, we finally point out that Brumer-Kramer refined Yoshida's conjecture in the paramodular case, conjecturing a bijection between isogeny classes of abelian surfaces $A/ \Q$ of conductor $M$ with $\End_{\Q}(A) = \Z$ and weight $2$ paramodular newforms of level $M$ with rational eigenvalues that are not Gritsenko lifts, with an equality of $L$-series (see \cite{BK}). Brumer-Kramer's conjecture has been verified numerically for small prime levels by Poor-Yuen \cite{PY}. As a complement, it would thus be interesting to see if our arguments can be adapted to the paramodular case.


\section{The $\Lambda$-adic Saito-Kurokawa lifting}\label{sectionone}

In this section, we construct an explicit 
$\Lambda$-adic Saito-Kurokawa lifting, using a key result of Stevens \cite{St} and following Kawamura \cite{Ka}. See also the anterior work \cite{Gu} by Guerzhoy. 

Let $f=\sum_{n=1}^\infty a_nq^n$ be a newform of weight 2 
and level $\Gamma_0(Np)$, where $N\geq 1$ is a fixed odd squarefree integer, 
$p\nmid N$ is an odd prime number and $a_n\in\Z$ for 
all $n\geq 1$. Consider the Hida family
$f_\infty(k)$ passing through $f$ introduced in
\eqref{Hida-family};   
recall that $k$ belongs to 
a neighborhood $U$ of 2 in the weight space $\mathcal X$,
with $U$ contained in the residue class of $2$ 
mod $p-1$, and  
$a_1(k)=1$ for all $k\in U$. 
For any integer $k\geq 2$ in $U$, let $f_k:=f_\infty(k)$ denote the 
$k$-specialization of 
$f_\infty$. We also let $\Lambda:=\Z_p[\![1+p\Z_p]\!]$ denote 
the Iwasawa algebra of $1+p\Z_p$ with coefficients in $\Z_p$. 

We first recall a well-known result of Stevens \cite{St} on the $\Lambda$-adic 
Shintani lifting. 
Let $\theta(f_k)\in S_{(k+1)/2}(4Np)$ denote the Shintani lifting of $f_k$, whose definition 
is recalled, for example, in \cite[\S 2.2]{St}; here $S_{(k+1)/2}(4Np)$ denotes the $\C$-vector space 
of modular forms of half integral weight $(k+1)/2$ and level $4Np$. 
In particular, $\theta(f_k)$ is well-defined
only up to a complex non-zero factor. 
A result of Shimura \cite{Sh} (cf. \cite[Proposition 2.3.1]{St}), asserts the
existence, for any positive even integer $k$, of a complex number $\Omega_{f_k}^-$ 
such that 
\[\frac{\theta(f_k)}{\Omega_{f_k}^-}\in \mathcal O_{f_k}\] where 
$\mathcal O_{f_k}$ denotes the ring of integers of the finite extension 
$K_{f_k}$ generated over $\Q$ be the Fourier coefficients of $f_k$.

\begin{rmk} \label{rem2.1}
$\Omega_{f_k}^-$ is defined in \cite[Theorem 4.8]{GS} in such a way that 
$\Omega_{f_k}^-\Phi_{f_k}$ (where $\Phi_{f_k}$ is the standard modular 
symbol associated to $f_k$ as in \cite[Definition 4.7]{GS}) spans the one-dimensional 
$K_{f_k}$-vector space consisting of ${\rm Sym}^{k-2}(K_{f_k}^2)$-valued $\Gamma_0(Np)$-invariant modular symbols
where the involution $\smallmat  {-1}001$ acts as $-1$ and the Hecke algebra acts through the character associated 
to $f_k$, cf. \cite[Section 4]{GS}.  \end{rmk}

We recall now the following result 
due to Stevens (see \cite[Theorem 3.3]{St}).

\begin{thm}[Stevens]\label{thm-stevens} There is a formal power expansion 
\[\Theta(k)=\sum_{n=1}^\infty\widetilde\gerb_n(k)q^n\] (where $\widetilde\gerb_n$ are $p$-adic analytic functions defined 
on $U$) and, for any positive even integer $k$, a $p$-adic number $\Omega_{k} \in\bar\Q_p$, 
such that:
\begin{enumerate}
\item $\Omega_2\neq 0$; 
\item 
\[\Theta(k)=\frac{\Omega_{k}}{\Omega_{f_{k}}^-}\theta(f_{k}).\]\end{enumerate}
\end{thm}

We shall use the above result to construct an explicit $\Lambda$-adic Saito-Kurokawa lifting. 
This is well-known to experts and follows in particular from a suitable correction of Li's thesis \cite{Li}, 
following Kawamura \cite{Ka}. We recall the constructions
which will be used later. 
For an  integer $k >2$ in $U$, let $F_k:=\SK_{Np}(f_k)$ denote the Saito-Kurokawa lifting of $f_k$ of level $Np$, 
the existence of which is proved in this generality in \cite{Ibu}. (Note that 
$F_k$ is is well-defined only up to a non-zero complex number.) The 
Siegel modular form $F_k$ has weight $k/2+1$ and level $\Gamma_0^2(Np)$ where, 
for any integer $M\geq 1$, we define  
\[\Gamma_0^2(M)=\left\{\smallmat ABCD\in{\rm Sp}_2(\Z)|C\equiv 0\mod M\right\}.\] 
The Fourier expansion of $F_k$ can be written as follows: 
\[
F_k (Z)=\sum_{T>0}\gera_T(k)q^T
\] where the sum is over all positive definite, half-integral 
symmetric matrices $T$ of size $2 \times 2$ and $q^T:=\exp\big(2\pi i \tr(TZ)\big)$. 
If we write 
\[g_k :=\theta(f_k)=\sum_{\text{$D\geq 1$; $(-1)^{k/2}D\equiv 0,1$ mod{4} }}\gerb_D(k)q^D
\] for the Shintani lifting of $f_k$ (which is well-defined only up to a non-zero 
complex number), 
then the image of the Saito-Kurokawa lifting 
satisfies the following relation, thanks to \cite[\S 3.4]{Ibu}: 
\begin{equation}\label{eq1bis}
\gera_T(k)=\sum_{d>0;\ d\mid c(T);\ (Np,d)=1}\gerb_{D_T/d^2}(k)d^{k/2}
\end{equation} where $D_T=\det(2T)$ and, for $T=\smallmat u{v/2}{v/2}w$,
$c(T):=\gcd(u,v,w)$. Of course, the only $d$'s contributing to the above sum 
are those for which $(-1)^{k/2}D_T/d^2\equiv 0,1$ mod 4. 
Define for $k\in U$ 
\[\widetilde\gera_T(k)=\sum_{d>0;\ d\mid c(T);\ (Np,d)=1}\widetilde\gerb_{D_T/d^2}(k).\] Then 
as an immediate corollary of Theorem \ref{thm-stevens} and the above discussion,
we find:

\begin{cor}[Li]\label{Li} The formal power series expansion: 
\[F_\infty(k):=\sum_{T>0}\widetilde\gera_{T}(k)q^T\] defined for $k\in U$ 
satisfies
\[F_\infty(k)=\frac{\Omega_k}{\Omega_{f_k}^-}F_k\] for any positive integer $k\in U$, where the 
$p$-adic numbers $\Omega_k$ 
are defined in Theorem \ref{thm-stevens}. 
\end{cor}

\begin{rmk}
In general, $F_k$ is not an ordinary form (in the sense that $F_k$ does not need to be an eigenform for $p$-Hecke operators 
such that the eigenvalues are $p$-adic units). 
One may ask for (semi-)ordinarity conditions by  
considering another family of Siegel modular forms interpolating $p$-adic stabilizations (in a suitable sense) 
of the Saito-Kurokawa lifting of the classical forms  
$f_k$ in the Hida family. This is the point of view taken in 
Kawamura's paper \cite{Ka}.  See 
\cite[Theorem 4.4]{Ka} for the construction of the $\Lambda$-adic family of Siegel modular 
forms interpolating suitable $p$-stabilizations of the Ikeda lifting of the classical forms $f_k$ in the 
Hida family.\end{rmk}

Instead of working with $F_k$, we will work with related forms $F_k^\sharp$, 
defined as follows. 
For an  integer $k >2$ in $U$, let $f_k^\sharp$ 
denote the unique newform of weight $k$ and level $\Gamma_0(N)$ such that 
\[f_k(z)=f_k^\sharp(z)-a_p(k)^{-1}p^{k-1}f_k^\sharp(pz).\]
We remark that, since $f$ is a newform, $f_k^\sharp$ is newform too
(see for example \cite[Corollary 1.3]{Hida} for details). 
Let $F_k^\sharp=\SK_N(f_k^\sharp)$ denote the Saito-Kurokawa lifting of level $N$ of $f_k^\sharp$, 
for an integer $k>2$ in $U$ (see again \cite{Ibu}, and note again that
$F_k^\sharp$ is well-defined only up to a non-zero complex number). The 
Siegel modular form $F_k^\sharp$ has weight $k/2+1$ and level $\Gamma_0^2(N)$, 
and its Fourier expansion can be written as
\[
F_k^\sharp(Z)=\sum_{T>0}A_T(k)q^T
\]  (same conventions as above). 
If we write 
\[g_k^\sharp:=\theta(f_k^\sharp)=\sum_{\text{$D\geq 1$; $(-1)^{k/2}D\equiv 0,1$ mod{4} }}c_D(k)q^D
\] for the Shintani lifting of $f_k^\sharp$, then the image of the Saito-Kurokawa lifting 
satisfies the following relation 
\begin{equation}\label{eq1}
A_T(k)=\sum_{d>0;\ d\mid c(T);\ (N,d)=1}c_{D_T/d^2}(k)d^{k/2}
\end{equation} (again, use \cite[\S 3.4]{Ibu}, with the same notations and conventions as above).
To use uniform notations, 
also denote by $c_D(2)$ instead of $\gerb_D(2)$ the Fourier coefficients of $g_2$, 
and by $A_T(2)$ instead of $\gera_T(2)$ the the Fourier coefficients of $F_2$. 

\begin{rmk} Equations \eqref{eq1} and \eqref{eq1bis} can be interpreted as an analogue 
for arbitrary level of the Maa\ss \ relations describing the image (called the Maa\ss \ Spezialschar) 
of the Saito-Kurokawa lifting in level 1. See \cite[Chap. 6]{EZ}, \cite[Proposition 3]{Ko}
and \cite{Ad} for precise references. For general levels, we refer to the discussion in \cite{Ibu} on the possibility 
of describing the image of the Saito-Kurokawa lifting in terms of such relations. 
\end{rmk}

To correct the asymmetry in \eqref{eq1} and \eqref{eq1bis} arising from the fact that
the condition $p\nmid d$ is in \eqref{eq1} but not in \eqref{eq2}, 
for integers $k\geq 4$ in $U$ we introduce still another family of Siegel modular forms
\[\widetilde F_k (Z)= F_k^\sharp(Z) - 
p^{k/2}F_k^\sharp(pZ).\] 
By \cite[Prop. 3.7]{Ibu}, we know that $\widetilde F_k$ is a modular form of weight $k/2+1$ and 
level $\Gamma_0^2(Np)$. 
Also, define for any positive integer $k\in U$: 
\begin{equation}\label{eq-final-version}
{A}^{(p)}_T(k) := \sum_{d>0;\ d\mid c(T);\ (d,Np)=1}c_{D_T/d^2}(k)d^{k/2}\end{equation} obtained from 
$A_T(k)$ by excluding the coefficients such that $p\mid c(T)$.
An easy computation shows that, for any integer $k\geq 2$ in $U$, 
\[
p^{k/2}F_k^\sharp(pZ)=\sum_{T>0;\ p\mid c(T);}\left(\sum_{p\mid d\text{ and }d\mid c(T)}c_{D_{T}/d^2}(k)d^{k/2}\right)q^T,\]
from which we have: 
\begin{equation}\label{eq2}
\widetilde F_k = \sum_{T>0} {A}_T^{(p)}(k) q^T.\end{equation} To use uniform notations, we will write $A^{(p)}_T(2)$ 
for $\gera_T(2)$. 
It follows immediately from \cite[\S 3.4]{Ibu} that 
$\widetilde F_k$ is the Saito-Kurokawa lifting of level $Np$ of the form $f_k^\sharp$, viewed as a modular form of level $Np$.  
Since the family of modular forms obtained in Corollary \ref{Li} above interpolates the modular forms $F_k$
for $k\in U$, we have 
$F_k\neq \widetilde F_k$ if $k\neq 2$. 
However, $F_k$ and $\widetilde F_k$
share interesting Fourier coefficients, as we will explain in the next section (see in particular Remark \ref{remark3.3}).

\section{Normalized Fourier coefficients}\label{sectiontwo}
In this section we modify the Fourier coefficients $A_T^{(p)}(k)$ introduced in the previous section 
in order to relate them with the work of Darmon-Tornar\'ia \cite{DT}.

We first recall some results on the coefficients $c_D(k)$.
For any prime 
number $\ell\mid Np$, denote by $w_\ell\in\{\pm 1\}$ the eigenvalue of the Atkin-Lehner involution $W_\ell$ 
acting on $f$. We first observe that, since $f_k^\sharp$ has trivial character, $g_k^\sharp$ is zero if and only if 
$k/2$ is even, 
and $f_2\neq 0$ (see for example \cite[page 137]{Sh}). 
Hence, if $c_D(k)\neq 0$ then $k/2\geq 1$ is odd and 
$-D\equiv 0,1$ modulo $4$. Also, by 
\cite[Corollary 1]{Ko2}, if $k>2$ and $k/2$ is odd, then $c_D(k)=0$ unless $\left(\frac{-D}{\ell}\right)=w_\ell$ for all 
primes $\ell\mid N$ and $c_D(2)=0$ unless $\left(\frac{-D}{\ell}\right)=w_\ell$ for all 
primes $\ell\mid Np$. 

Motivated by the above discussion, choose   
a fundamental discriminant $D_0$ such that $c_{D_0}(2)\neq 0$ and let $D>0$ be an integer
such that $p\nmid D$ and $\left(\frac{-D}{\ell}\right)=w_\ell$ for all 
primes $\ell\mid N$, so that the family $c_D(k)$ need not vanish identically for all $k$ 
(but note that $c_D(2)$ vanishes if $\left(\frac{-D}{p}\right)=-w_p$, for example). 

\begin{prop}\label{prop1} Let $D$ and $D_0$ be as above. 
There is a $p$-adic neighborhood $\mathcal U\subseteq U$ of $2$ such that:
\begin{enumerate} 
\item $c_{D_0}(k)\neq 0$ for $k\in\Z\cap\mathcal U$ with $k\geq 2$.
\item The function $k\mapsto \widetilde c_D(k)$, defined for integers $k\geq 4$ in $U$, 
by:  \[\widetilde c_D(k):=  \frac{\left(1-\left(\frac{-D}{p}\right)a_p(k)^{-1}p^{k/2-1}\right)c_D(k)}
{\left(1-\left(\frac{-D_0}{p}\right)a_p(k)^{-1}p^{k/2-1}\right)c_{D_0}(k)}=\frac{c_{Dp^2}(k)}{c_{D_0p^2}(k)}\]
extends to a $p$-adic analytic function on $\mathcal U$ satisfying: 
\[\widetilde c_D(2)=\frac{c_{Dp^2}(2)}{c_{D_0p^2}(2)}=\frac{c_D(2)}{c_{D_0}(2)}.\]
\end{enumerate}
\end{prop}

\begin{rmk} \label{rem3.2}  This is \cite[Proposition 1.3]{DT}. We simply note that,  
since $p\nmid D$, this can be interpreted as a $p$-stabilization process. More precisely, for $k\geq 4$ an 
even integer 
(and $k/2$ odd, otherwise $\widetilde c_D(k)$ vanishes)  we have 
\[\widetilde c_D(k)=\frac{\gerb_D(k)}{\gerb_{D_0}(k)}=\frac{\widetilde\gerb_D(k)}{\widetilde\gerb_{D_0}(k)}.\]
This follows from the explicit description of Hecke operators given, for example, 
in  \cite[Ch. IV, Prop. 13]{Kob}.\end{rmk}

In order to relate our coefficients $A_T^{(p)}(k)$ with the coefficients $\widetilde c_D(k)$ appearing 
in Prop. \ref{prop1},  we proceed as follows. 
First recall from the above discussion that $c_{D}(k)=0$ for all $D$ such that $-D\equiv 2,3$ mod 4 
and for all positive integers $k\in\mathcal U$. If we define 
\[S_T:=\{{\text{$d\in\Z$ such that $d\mid c(T)$,  $(d,Np) = 1$, $-D_T/d^2\equiv 0,1$ mod 4}}\},\] 
we have
\begin{equation}\label{Fourier}
A_T^{(p)}(k)=\sum_{d\in S_T}d^{k/2}c_{D_T/d^2}(k).\end{equation}

Fix now a fundamental discriminant $\gerd$ with $\gerd<0$ and an integer $n\neq 0$. 
Define, for any even integer $k\geq 4$ with ${k/2}$ odd,
\[
\rho_{\gerd,n}(k):={\sum_{d\mid n;\ (d,N)=1}\mu(d)\left(\frac{\gerd}{d}\right)d^{k/2-1}a_{n/d}(k)}
,\] where $\mu(d)$ is the Mšbius function.  
Thanks to 
\cite[(11)]{Ko2}, for all even integers $k>2$ in $\mathcal U$ we have: 
\begin{equation}\label{eq-rho-bis}c_{|\gerd| n^2}(k)=c_{|\gerd|}(k)\rho_{\gerd,n}(k).\end{equation}

We now use \eqref{eq-rho-bis} to simplify the above sum \eqref{Fourier}.  
For this, we need the following preliminary discussion. For any integer $D$ we can write 
$D=\gerd\gerf^2$ 
with $\gerd$ the fundamental discriminant of the quadratic 
field $\Q(\sqrt{\gerd})$ and $\gerf>0$ a half-integer. 
In this case, we say that $\gerd$ is 
the fundamental discriminant associated with $D$. 
Also, note that $\gerf\in\Z$ if $D\equiv0,1\mod 4$. More precisely, 
if we denote by $\delta$ the maximal integer such that $\delta^2\mid D$, then $\gerd=D/\delta^2$ and 
$\gerf=\delta$ unless $D\equiv 0$ mod 4 and $D/\delta^2\equiv 2,3$ mod 4, in which case
$\gerd=D/(\delta/2)^2$ and
$\gerf=\delta/2$.  
If $D\equiv 2,3$ mod 4, then 
$\gerd=4D/\delta^2$ and $\gerf=\delta/2\not\in\Z$. 
More generally, let $d\mid\delta$, so that $d^2\mid D$. 
Write $D=\gerd_D\gerf^2_D$ and $D/d^2=\gerd_{D/d^2}(\gerf_{D/d^2})^2$, with $\gerd_D$ and $\gerd_{D/d^2}$ 
the fundamental discriminants associated with $D$ and $D/d$, respectively. First note that the maximal integer 
whose square divides $D/d^2$ is $\delta/d$. We then have 
$\gerd_D=\gerd_{D/d^2}$ and $\gerf_D=d\gerf_{D/d^2}$ unless
$D\equiv 0$ mod 4, $D/\delta^2\equiv 1$ mod 4 and $D/d^2\equiv 2,3$ mod 4, in which case $4\gerd_D=\gerd_{D/d^2}$ 
and $\gerf_D=2d\gerf_{D/d^2}$. 

As anticipated, we now use \eqref{eq-rho-bis} to simplify \eqref{Fourier}.  For any matrix $T>0$ write 
\begin{equation}\label{D_T}-D_T=\gerd_T\gerf_T^2,\end{equation} with $\gerd_T<0$. 
We have $-D_T\equiv 0,1$ mod 4,  so, in 
particular, the above discussion shows that $\gerf_T\in\Z$.  
We thus have 
\[c_{D_T}(k)=c_{|\gerd_T|\gerf_T^2}(k)=c_{|\gerd_T|}(k)\rho_{\gerd_T,\gerf_T}(k)\] 
for all integers $k\in\mathcal U$ with $k>2$. More generally, 
fix $d\in S_T$. For $T=\smallmat u{v/2}{v/2}w$, define 
$T/d:=\smallmat {u/d}{v/2d}{v/2d}{w/d}$. Then $D_{T/d}=D_T/d^2$. Under the assumption 
that $-D_T/d^2\equiv 0,1$ mod 4, the above discussion shows that $\gerf_{T/d}\in\Z$ 
and that $\gerd_{T}=\gerd_{{T/d}}$ and $\gerf_{T}=d\gerf_{{T/d}}$. We can thus write: 
\begin{equation}\label{eq4}
c_{D_T/d^2}(k)= c_{D_{T/d}}(k)=c_{|\gerd_{T/d}|\gerf_{T/d}^2}(k)=
c_{|\gerd_T|}(k)\rho_{\gerd_T,\gerf_T/d}(k)\end{equation}
for all integers $k\in\mathcal U$ with $k>2$. Inserting \eqref{eq4} into \eqref{Fourier} we get
\begin{equation}\label{++}
A_T^{(p)}(k)=c_{|\gerd_T|}(k)\sum_{d\in S_T}d^{k/2}\rho_{\gerd_T,\gerf_T/d}(k).\end{equation}
for all integers $k\in\mathcal U$ with $k>2$. 

We are now ready to define our modified Fourier coefficients. 
Fix a matrix $T_0$ such that $c(T_0)=1$ and $c_{D_{T_0}}(2)\neq 0$, 
which exists by a combination of the explicit formula of \cite{KZ} relating $c_m(2)$ 
with special values of $L$-series and non-vanishing results for $L$-series 
by \cite{BFH}, \cite{Wa} 
(see for example \cite[Lem. 2.4]{Ka} for details).
Define, for integers $k$ in $\mathcal U$ with $k>2$, \begin{equation}\label{defA}
\widetilde A_T(k):= \frac{\left(1-\left(\frac{\gerd_{T}}{p}\right)a_p(k)^{-1}p^{k/2-1}\right) {A}^{(p)}_{T}(k)}
{\left(1-\left(\frac{D_{T_0}}{p}\right)a_p(k)^{-1}p^{k/2-1}\right){A}^{(p)}_{T_0}(k)}.\end{equation}

To simplify notations, define, for all integers $k\in\mathcal U$ with $k>2$: 
\begin{equation}\label{DEF-n-T} 
n_T(k):=\sum_{d\in S_T}d^{k/2}
\rho_{\gerd_T,\gerf_T/d}(k).\end{equation}
Putting \eqref{++} into \eqref{defA} we get, for all integers $k\in\mathcal U$ with $k\geq 4$: 
\begin{equation}\label{EQ-A-TILDE}
\widetilde A_T(k)= \frac{\left(1-\left(\frac{\gerd_{T}}{p}\right)a_p(k)^{-1}p^{k/2-1}\right)c_{|\gerd_T|}(k)}
{\left(1-\left(\frac{D_{T_0}}{p}\right)a_p(k)^{-1}p^{k/2-1}\right)c_{D_{T_0}}(k)}n_T(k)=
\widetilde c_{|\gerd_T|}(k)n_T(k).\end{equation}

\begin{rmk}\label{remark3.3}Since $c_{D_{T_0}}(2)\neq 0$, clearly $p\nmid D_{T_0}$. Suppose that $p\nmid D_T$. Then 
\begin{equation}\label{EQ-REM}
\widetilde A_T(k)=\frac{\gera_T(k)}{\gera_{T_0}(k)}.\end{equation}
To show this, combining \eqref{DEF-n-T} and \eqref{EQ-A-TILDE} we find 
\[\widetilde A_T(k)=
\sum_{d\in S_T}\frac{\left(1-\left(\frac{\gerd_{T}}{p}\right)a_p(k)^{-1}p^{k/2-1}\right)c_{|\gerd_T|}(k)
\rho_{\gerd_T,\gerf_T/d}(k)}
{\left(1-\left(\frac{D_{T_0}}{p}\right)a_p(k)^{-1}p^{k/2-1}\right)c_{D_{T_0}}(k)}d^{k/2}.\]
Using \eqref{EQ-A-TILDE} we find 
\[\widetilde A_T(k)=
\sum_{d\in S_T}\frac{\left(1-\left(\frac{\gerd_{T}}{p}\right)a_p(k)^{-1}p^{k/2-1}\right)c_{D_{T}/d^2}(k)}
{\left(1-\left(\frac{D_{T_0}}{p}\right)a_p(k)^{-1}p^{k/2-1}\right)c_{D_{T_0}}(k)}d^{k/2}.\]
Since $p\nmid D_T$, then in particular $p\nmid \gerf_T$. 
So we have 
$\left(\frac{\gerd_{T}}{p}\right)=\left(\frac{\gerd_T\gerf^2_T}{p}\right)=\left(\frac{-D_T}{p}\right)$. 
Further, if $d\in S_T$ then $p\nmid d$ and we have 
$\left(\frac{-D_T}{p}\right)=\left(\frac{-D_{T}/d^2}{p}\right)$. Hence we obtain the relation 
\[\widetilde A_T(k)=
\sum_{d\in S_T}\frac{\left(1-\left(\frac{-D_T/d^2}{p}\right)a_p(k)^{-1}p^{k/2-1}\right)c_{D_{T}/d^2}(k)}
{\left(1-\left(\frac{D_{T_0}}{p}\right)a_p(k)^{-1}p^{k/2-1}\right)c_{D_{T_0}}(k)}d^{k/2}\]
and, using the definition of $\widetilde c_D(k)$, we get 
\begin{equation}\label{EQ-REM-1}\widetilde A_T(k)=
\sum_{d\in S_T}\widetilde c_{D_T/d^2}(k)d^{k/2}.\end{equation} Now, by 
Remark \ref{rem3.2}, we have 
\begin{equation}\label{EQ-REM-2}
\frac{\widetilde\gera_T(k)}{\widetilde\gera_{T_0}(k)}
\frac{\gera_T(k)}{\gera_{T_0}(k)}=\sum_{d\in S_T}\frac{\widetilde\gerb_{D_T/d^2}(k)}
{\widetilde\gerb_{D_{T_0}}(k)}d^{k/2}=
\sum_{d\in S_T}\widetilde c_{D_T/d^2}(k)d^{k/2}.\end{equation}
Combining \eqref{EQ-REM-1} and \eqref{EQ-REM-2} gives \eqref{EQ-REM}.\end{rmk}

\begin{prop} \label{prop2}
The function $k\mapsto \widetilde A_T(k)$, defined for integers $k>2$ in $\mathcal U$, 
extends to a $p$-adic analytic function on $\mathcal U$. 
Moreover, for $k=2$ we get the equality:
\begin{equation}\label{eq-A-T}\widetilde A_T(2)= n_T\cdot\widetilde c_{|\gerd_T|}(2)=n_T\cdot\frac{c_{|\gerd_T|}(2)}{c_{D_{T_0}}(2)} \end{equation}
where $n_T:=n_T(2)$. 
\end{prop}

\begin{proof} Clearly, $\rho_{\gerd,n}(k)$, defined for integers $k>2$ in $\mathcal U$, 
can be extended analytically to all of $\mathcal U$ because the same is true for the Fourier coefficients $a_n(k)$.  
We denote by $\widetilde \rho_{\gerd,n}(k)$ the resulting function. Likewise, we set 
$\widetilde n_T(k)$ the obvious analytic extension of $n_T(k)$ obtained from $\widetilde\rho_{\gerd_T,\gerf_T/d}(k)$ 
for all $d\in S_T$. We can thus define for all $k\in\mathcal U$ 
\begin{equation}\label{def-tildeA}
\widetilde A_T(k):=\widetilde c_{|\gerd_T|}(k)\widetilde n_T(k)\end{equation} which is, thanks to Proposition \ref{prop1}, 
the sought for extension of $\widetilde A_T(k)$. 
The value at $2$ is then a consequence of the second part of Proposition \ref{prop1}. 
\end{proof}

\begin{rmk}\label{remark-n_T}
The integer $n_T$ above can be explicitly written as 
\[n_T=\sum_{d\in S_T}{\sum_{e\mid \gerf_T/d;\ (e,N)=1}d\cdot\mu(e)\cdot\left(\frac{\gerd_T}{e}\right)\cdot a_{\gerf_T/(de)}(2)}.\]
In level $1$, the analogous sum is seen to be non-zero by writing it explicitly as a finite product of non-zero terms arising as special values of local singular series polynomials.
That is, in the terminology of \cite[Thm. 4.1]{Ka}, we have:
\[\begin{split}A_T( \Lift(\phi)^*) = \Big( 1 - \Big( \frac{\gerd_T}{p} \Big) \beta_p(\phi) p^{-r} \Big)& c_{|\gerd_T|}(\psi) \times \alpha_p(\phi)^{v_p(\gerf_T)+ 2}\times\\
&\times \prod_{\ell | \gerf_T, \ell \neq p}  \alpha_{\ell}(\phi)^{v_{\ell}(\gerf_T)} F_{\ell}(T; \beta_{\ell}(f) \ell^{-r-1}).\end{split}\]
Here $\phi$ is a $p$-ordinary normalized Hecke eigenform on $\SL_2(\Z)$ of weight $2r$, ${\rm Lift}(\phi)^*$ 
is a certain $p$-stabilization of the Saito-Kurokawa lifting of $\phi$, with Fourier coefficients $A_T({\rm Lift}(\phi)^*)$
(introduced in \cite[(15)]{Ka}), 
$\psi$ is the Shintani lifting of $\phi$ with Fourier coefficients $c_D(\psi)$,  
$\alpha_p(\phi)$ and $\beta_p(\phi)$ are the roots of 
the Hecke polynomial at $p$ of $\phi$, ordered in such a way that $\alpha_p(\phi)$ is a $p$-adic unit, $v_p$ and $v_\ell$ 
denote $p$-adic and $\ell$-adic valuations, $F(T; X)$ is a polynomial introduced in \cite[(eq 4)]{Ka},
and all other symbols have the same meaning as above.  
We do not know a proof of the analogous non-vanishing statement for higher levels.\end{rmk} 

\section{Global points on elliptic curves}\label{sectionthree}
In this section, we compute the $p$-adic derivative of the explicit formula for the Fourier coefficients of the Saito-Kurokawa family exhibited in Section \ref{sectiontwo}, and relate it to global points on elliptic curves via the fundamental results of \cite{BD1} and \cite{BD}.

\subsection{Global points associated to the Shintani lifting}\label{sec4.1}
We start by recalling the work of 
Darmon and Tornar\'ia \cite{DT}, for which we need to introduce some extra notations.  
Recall that $g=g_2$ denotes the Shintani lifting of $f=f_2$, as usually well-defined only up to scalars, with 
Fourier coefficients $c_D(2)$. Recall that 
$w_\ell\in\{\pm 1\}$ is the eigenvalue of the Atkin-Lehner involution $W_\ell$ 
acting on $f$, for a number $\ell\mid Np$. 
In accordance with \cite{DT}, we introduce the following terminology:  
\begin{enumerate}
\item[(I)] We call \emph{discriminants of type} (I) those fundamental discriminants $\gerd<0$ 
such that 
\[\left(\frac{\gerd}{\ell}\right)=w_\ell\] for all primes $\ell\mid Np$. 
\item[(II)] We call \emph{discriminants of type} (II) those $\gerd$ such that 
\[\left(\frac{\gerd}{\ell}\right)=w_\ell\] for all primes $\ell\mid N$ and 
\[\left(\frac{\gerd}{p}\right)=-w_p.\] 
\end{enumerate}
If $\gerd$ is of type II, then $c_\gerd(2)=0$ and then $\widetilde c_\gerd(2)=0$ too. 
In this latter case, it becomes interesting to look 
at the value at $2$ of the derivative of $c_\gerd(k)$ with respect to $k$. 

Let $E/\Q$ denote the elliptic curve of conductor $Np$ associated with $f$ via the Eichler-Shimura construction. 
Let 
$\gerd<0$ be a fundamental discriminant and, to simplify notations, define the imaginary quadratic field
\[K_\gerd:=\Q( \sqrt{\gerd}).\]
Denote by 
$E(K_\gerd)^-$ the submodule of the Mordell-Weil group $E(K_\gerd)$ on which the non-trivial 
involution of $\Gal(K_\gerd/\Q)$ acts as $-1$.  

Since $p$ does not divide $N$, the curve $E$ has multiplicative reduction at $p$. 
For a fundamental discriminant $\gerd<0$, 
let $L(f,\chi_{\gerd},s)$ denote the complex $L$-function of $f$ twisted by 
the quadratic character \[\chi_{\gerd}(n):=\left(\frac{\gerd}{n}\right)\] of $K_\gerd$. 
So, if the real part of $s$ is 
sufficiently large, we have
\[L(f,\chi_{\gerd},s)=\sum_{n=1}^\infty\chi_{\gerd}(n)a_nn^{-s}.\] 

We now state the main result of Darmon-Tornar\'ia (\cite[Theorem 1.5]{DT}). Recall that
$\Phi_{\rm Tate}$ denotes Tate's uniformization introduced in
\eqref{Phi-Tate} and $\log_E$ is the logarithm defined in \eqref{log}.

\begin{thm}[Darmon-Tornar\'ia]\label{thm-DT}
Let $N>1$ and suppose that $\gerd<0$ is a fundamental discriminant of type II. 
Then there exists an element $P_\gerd\in E(K_\gerd)^-\otimes_\Z\Q$  such that 
\begin{enumerate}
\item $\frac{\partial}{\partial k}\widetilde c_{|\gerd|}(k)_{|k=2}=\log_E(P_\gerd)$.
\item $P_\gerd\neq 0$ if any and only if $L'(E,\chi_{\gerd},1)\neq 0$. 
\end{enumerate}
\end{thm}

\begin{rmk}The point $P_\gerd$ in Theorem \ref{thm-DT} comes, as already mentioned, from the theory of Darmon points. 
These points were introduced in \cite{Dar}, under the name of Stark-Heegner points,  
as local points on elliptic curves, and are conjectured to be global points, defined over ring class fields of real quadratic fields. 
A special case of this conjecture, which is needed in Theorem \ref{thm-DT} above, is proved 
by Bertolini and Darmon  in \cite{BD}, using results from \cite{BD1}, by establishing a connection between classical Heegner points and Darmon points.  
For more details, we refer the reader to the discussion in 
\cite[Section 3]{DT}. \end{rmk}

For the next result, let $\gerd'<0$ be a fundamental discriminant of type I, prime to $\gerd$,
and define $\Delta(\gerd,\gerd'):=
\gerd\gerd'$. 
In the following lines, we simply write $\Delta$ when the role of $\gerd$ and $\gerd'$ are clear from the context. 
In particular, since $\gerd$ is of type II, $\Delta$ is not a perfect square and is prime to $Np$.
Let $\chi_\Delta$ denote the genus character associated to the pair 
of quadratic Dirichler characters $\chi_\gerd$ and $\chi_{\gerd'}$ of $\Q(\sqrt{\gerd})$ and $\Q(\sqrt{\gerd'})$. Then 
$\Delta$ is the discriminant of the real quadratic field $\Q(\sqrt{\Delta})$ and the field $\Q(\sqrt{\gerd},\sqrt{\gerd'})$ 
cut out by $\chi_\Delta$ 
is a quadratic extension of $\Q(\sqrt{\Delta})$ (unless $\chi_\Delta$ is trivial, in which case coincides with $\Q(\sqrt{\Delta})$). 
Since all primes $\ell\mid N$ are split in $K$, we may choose 
an integer $\delta$ such that $\delta^2\equiv \Delta$ mod $4N$. 
Recall from \cite[\S 2]{DT} that a primitive binary quadratic form $Q(x,y)=Ax^2+Bxy+Cy^2$
of discriminant $\Delta$ is said to be a Heegner form relative to the level $N$ if $N\mid A$ and $B\equiv \delta$ mod $N$.
Let $\mathcal F_\Delta$ denote the set of such forms. Let $a+b\sqrt \Delta$ denote a fundamental unit of norm one
in $\Z[(\Delta+\sqrt\Delta)/2]$, normalized such that $a>0$ and $b>0$, 
and, for $Q(x,y)=Ax^2+Bxy+Cy^2\in\mathcal F_\Delta$, define the matrix
\[\gamma_Q=\mat {a+bB}{2Cb}{-2Ab}{a-bB}\in\Gamma_0(N).\]

The group $\Gamma_0(N)$ acts on $\mathcal F_\Delta$ from the right by the formula 
\[(Q|\gamma)(x,y)=Q(ax+by,cx+dy)\] for $\gamma=\smallmat abcd$. 
Let $d=Q(m,n)$ be any integer represented by $Q$ (namely, such that there are integers $(m,n)$ such that 
$Q(m,n)=d$) satisfying the condition 
$\gcd(d,\gerd)=1$. The genus character $\chi_\Delta$ defines a function, denoted by the same symbol, 
\[\chi_\Delta:\mathcal F_\Delta/\Gamma_0(N)\longrightarrow \{\pm1\},\qquad \chi_\Delta(Q)=
\left\{\begin{array}{llcc}0&\text{if $\gcd(A,B,C,\gerd)>1$}\\ 
\left(\frac{\gerd}{d}\right) &\text{otherwise.}\end{array}\right.\] 

For any integer $k\in\mathcal U$ with $k\geq 4$, we consider the integrals 
\[I_\C(f_k^\sharp,P,r,s):=\int_r^sf_k^\sharp(z)P(z)dz\] where $P$ is a polynomial of degree at most $k-2$ with 
coefficients in $\C$, $r,s$ are in $\mathbb P^1(\Q)$ and the above integral is along any path in the upper 
half plane connecting $r$ and $s$. A result of Shimura shows that one can find a complex period $\Omega_{f_k^\sharp}^-$, 
analogous to $\Omega_{f_k}^-$ considered in Remark \ref{rem2.1}, such that 
\[I(f_k^\sharp,P,r,s):=I_\C(f_k^\sharp,P,r,s)/\Omega_{f_k^\sharp}^-\] belongs to the field generated over $\Q$ by the Fourier 
coefficients of $f_k^\sharp$, for all $P$, $r$ and $s$. Define
\[J(f_k^\sharp,P,r,s):=(1-a_p(k)^{-2}p^{k-2})I(f_k^\sharp,P,r,s).\] 

\begin{rmk}
The only choice among the complex periods $\Omega_{f_k^\sharp}^+$ and $\Omega_{f_k^\sharp}^-$ 
which is relevant for this paper is $\Omega_{f_k^\sharp}^-$. This corresponds to the choice made 
in Theorem \ref{thm-stevens} (see also Remark \ref{rem2.1}).  In fact, the proof of Theorem \ref{thm-stevens} 
proceeds by $p$-adically interpolating the above integrals $I(f_k^\sharp,P,r,s)$, which are eventually related 
to Fourier coefficients of half-integral modular forms by Shintani's work. However, 
the choice of periods $\Omega_{f_k^\sharp}^+$ instead of $\Omega_{f_k^\sharp}^-$ is possible 
and leads to a parallel theory of Darmon points in this situation 
and conjectural relations with half-integral modular forms: see
\cite[Secs. 4 and 5]{DT} for details.  
 \end{rmk}

We fix now embeddings $\bar\Q\hookrightarrow \C$ and $\bar\Q\hookrightarrow\C_p$.
By \cite[Proposition 3.4]{BD}, one knows that, 
for any choice of $\tau\in \Q_{p^2} \backslash \Q_p$, where $\Q_{p^2}$ is the quadratic unramified extension of $\Q_p$, the function 
$k\mapsto J(f_k^\sharp,P,r,s)$, a priori only defined for integers $k>2$ in $\mathcal U$, 
extends to a $p$-adic analytic function on $\mathcal U$ that vanishes at $k=2$. We shall denote $k\mapsto J(k,P,r,s)$ 
this function. 

Fix a square root $\sqrt\Delta$ of $\Delta$ in $\Q_{p^2}$ and, for $Q\in\mathcal F_\Delta$ as above,
define \[\tau_Q:=\frac{-B +\sqrt{\Delta}}{2A}.\] We may then define
\[J(f,Q):=\frac{d}{dk}J(k,(z-\tau_Q)^{k-2},r,\gamma_Q(r))_{|k=2}.\] Twisting $J(f,Q)$ by $\chi_\Delta$, we may 
also define: 
\begin{equation}\label{shintani}
J(f,\gerd,\gerd')=\sum_{Q\in\mathcal F_\Delta/\Gamma_0(N)}\chi_\Delta(Q)J(f,Q).\end{equation}
The second result we quote from the work of Darmon-Tornar\'ia is \cite[Theorem 5.1]{DT}, which can be restated
as follows:  
\begin{thm}[Darmon-Tornar\'ia]\label{DT2} Fix $\gerd<0$ is a fundamental discriminant of type II. 
\begin{enumerate}
\item For any fundamental discriminant $\gerd'$ of type I, we have the relation: 
\[\widetilde c_{|\gerd'|}(2)\log_E(P_\gerd)=J(f,\gerd,\gerd').\]
\item 
If $N>1$ then the function $\gerd'\mapsto J(f,\gerd,\gerd')$ is non-zero if and only if $L'(E,\chi_\gerd,1)\neq0$.\end{enumerate}\end{thm}

\subsection{Global points associated to the Saito-Kurokawa lifting}

We combine the explicit $\Lambda$-adic Saito-Kurokawa lifting with Darmon-Tornar\'ia's result. 
We first introduce the $L$-function attached to $\SK(f)$. 
Denote by $L(\SK(f),s)$ the Adrianov $L$-function associated with $\SK(f)$, whose definition can be found, 
e. g., in \cite[p. 179]{MRV}, where it is denoted by  $Z^*_{F}(s)$ for $F=\SK(f)$. 
We have the following relation between $L(\SK(f),s)$ and the standard $L$-function of $f$ (see \cite[Theorem 8]{MRV})  
\[
L({\SK(f)},s)=\zeta(s)\zeta(s-1)L(f,s) 
\] from which it is apparent that the central critical point is $s=1$.
Accordingly, we may consider the 
$\chi_{\gerd_T}$-twisted $L$-functions 
\[
L(\SK(f),\chi_{\gerd_T},s)=L(\chi_{\gerd_T},s)L(\chi_{\gerd_T},s-1)L(f,\chi_{\gerd_T},s). 
\] Since $\gerd_T$ is of type II, $L(f,\chi_\gerd,1)=0$ and we have 
\[L'(\SK(f),\chi_\gerd,1)=L(\chi_\gerd,0)L(\chi_\gerd,1)L'(f,\chi_\gerd,1) \] 
In particular, since the factor $L(\chi_\gerd,0)L(\chi_\gerd,1)$ is non-zero, we see that 
\begin{equation}\label{derivatives}
L'(\SK(f),\chi_\gerd,1)\neq 0\Longleftrightarrow L'(f,\chi_\gerd,1)\neq 0. 
\end{equation}

Before stating the main result of this paper, let us recall 
the integer $n_T$ introduced in Proposition \ref{prop2} and define the point 
\[Q_T:=n_T\cdot P_{\gerd_T}.\]
Suppose from now on that $n_T \neq 0$.

\begin{thm}\label{thm1} Let $N>1$. Suppose that $-D_T = \gerd_T \cdot \gerf^2_T$, where $\gerd_T<0$ is a fundamental discriminant of type II.
Then:
\begin{enumerate} \item We have the following relation: 
\[  \frac{\partial}{\partial k}\widetilde A_T(k)_{|k=2} = 
\log_E Q_{T}.\] 
\item $Q_T $ is non-zero in $E(K_{\gerd_T})^-\otimes_\Z\Q$ if and only if $L'(\SK(f),\chi_{\gerd_T},1)\neq 0$.\end{enumerate}
\end{thm}

\begin{proof}
Using \eqref{def-tildeA} to compute formally the $p$-adic derivative of the Fourier coefficients $\widetilde{A}_T(k)$, we find: 
\[
\frac{\partial}{\partial k}\widetilde A_T(k)_{|k=2} =
\frac{\partial}{\partial k}\widetilde c_{|\gerd_T|}(k)_{|k=2} \cdot\widetilde n_T(2)+\widetilde c_{|\gerd_T|}(2)\cdot
\frac{\partial}{\partial k}\widetilde n_T(k)_{|k=2} .\]
Since $\gerd_T$ is of type II, we have $\widetilde c_{|\gerd_T|}(2)=0$, and thus Theorem \ref{thm-DT}
implies the result, in light of the definition of $Q_T$. For the
second part, simply note that if $P_{\gerd_T}\neq 0$ then any non-zero multiple of it is non-zero and then use the relation \eqref{derivatives} together with Theorem \ref{thm-DT}, Part (2). 
\end{proof}

\begin{thm}\label{thm2} Fix $T$ such that 
$-D_T = \gerd_T \cdot \gerf^2_T$ with $\gerd_T<0$ a fundamental discriminant of type II.
For $T'$ such that $-D_{T'} = \gerd_{T'} \cdot \gerf^2_{T'}$ with $\gerd_{T'}$ a 
fundamental discriminant of type I, we have: 
\begin{enumerate}
\item $\widetilde A_{T'}(2)\log_E(P_{\gerd_T})= n_{T'}J(f,\gerd_T,\gerd_{T'})$.
\item If $N>1$   
then the function $\gerd_{T'}\mapsto J(f,\gerd_T,\gerd_{T'})$ is non-zero 
if and only if $L'(\SK(f),\chi_{\gerd_T},1)\neq0$.\end{enumerate}
\end{thm}

\begin{proof} Combining \eqref{eq-A-T} and Theorem \ref{DT2}, we have 
\[ \widetilde A_{T'}(2)\log_E(P_{\gerd_T})=
\widetilde c_{|\gerd_{T'}|}(2)n_{T'}\log_E(P_{\gerd_T})=
n_{T'}J(f,\gerd_T,\gerd'_T).\] This shows the first part, while the second 
is just a restatement of the second part of Theorem \ref{DT2} combined with \eqref{derivatives}. 
\end{proof}

{\bf Acknowledgments:} 
\noindent
We would like to thank: H. Kawamura for kindly contributing his expertise on the Saito-Kurokawa lifting, and for immediately pointing out a mistake in a preliminary version of this paper; H. Katsurada for mentioning \cite{Ibu} and thus providing the correct tool the rectify the above mistake; T. Ibukiyama for sending us the preprint \cite{Ibu} before its public distribution; and the Hakuba conference organizers of 2007 and 2008 for giving the second author the chance to learn about explicit formulae for classical automorphic forms from the experts. This work was entirely completed while the second author (and for a short stay, the first author) enjoyed the hospitality of the Max Planck Institute fŸr Mathematik in Bonn. Finally, we would like to thank the referee for his/her careful reading of the manuscript and for valuable suggestions and comments.

\end{document}